\documentclass[12pt,twoside]{article} 
\usepackage{amsmath, amssymb, amsthm, amsfonts, amsxtra, latexsym, amscd,
pb-diagram,  graphics,rotating,xy}

\date{} 

\begin{document}

\centerline{} 

\centerline{} 

\centerline {\Large{\bf Duals of Ann-categories}} 



\centerline{} 

\centerline{\bf {Nguyen Tien Quang and Dang Dinh Hanh}} 

\centerline{} 

\centerline{Department of Mathematics, Hanoi National University of Education} 

\centerline{ 136 Xuan Thuy Street, Hanoi, Vietnam} 

\centerline{nguyenquang272002@gmail.com  \ \  \ \ \   ddhanhdhsphn@gmail.com} 

\newtheorem{Theorem}{\quad Theorem}[section] 

\newtheorem{Proposition}[Theorem]{\quad Proposition}

\newtheorem{Definition}[Theorem]{\quad Definition} 

\newtheorem{Corollary}[Theorem]{\quad Corollary} 

\newtheorem{Lemma}[Theorem]{\quad Lemma} 

\newtheorem{Example}[Theorem]{\quad Example} 

\begin{abstract} 
Dual  monoidal category $\mathcal C^\ast$ of a  monoidal functor $F:\mathcal C\rightarrow \mathcal V$ has been constructed by S. Majid. In this paper, we extend the construction of dual structures for an Ann-functor $F:\mathcal B\rightarrow \mathcal A$. In particular, when $F=id_{\mathcal A}$, then the dual  category $\mathcal A^{\ast}$ is indeed the center of $\mathcal A$ and this is a braided Ann-category.
\end{abstract} 



\section{Introduction}
Categories with quasi-symmetry appeared under the heading ``braided monoidal categories" in a connection with low dimensional topology \cite{JS1},  as well as in the context of quantum groups \cite{Kas}.

The concept ``{\it dual of monoidal category}" appeared in \cite{Maj} in the following case. The Hopf algebra can be built via a monoidal category $\mathcal C$ and a functor $F: \mathcal C\rightarrow {\bf Vec}$. This event can be generalized as ${\bf Vec}$ is replaced by a monoidal category $\mathcal V$. Now, if $F$ is a monoidal functor, then $\mathcal C$ is called {\it functored} on $\mathcal V$, or $(\mathcal C,F)$ is called a {\it $\mathcal V$-category} in A. Grothendieck's terminology \cite{Gro}. In this situation, S. Majid built the monoidal category $(\mathcal C, F)^\ast = (\mathcal C^\ast,  F^\ast)$, named ``{\it full dual category}" of $ (\mathcal C, F)$. The objects of $(\mathcal C, F)^\ast$  are pairs $(V, u_V)$, consisting of $V\in\mathcal C$ and  a natural transformation $u_V = (u_{V,X}:V\otimes FX\rightarrow FX\otimes V)$ satisfying the compatition with the monoidal functor $(F, \widetilde{F})$. The full subcategory $(\mathcal C, F)^\circ$ consists of objects  $(V, u_V)$ where $u_{V,X}$ are isomorphisms. It is interesting when   $\mathcal V = \mathcal C$ and $F = id$, then  $(\mathcal C, F)^\circ$ is a braided monoidal category, called  the {\it center}  $Z(\mathcal C)$ of the monoidal category $\mathcal C$.

The notion of the center of a monoidal category appeared first in  \cite{JS1}, \cite{Maj}. It was a construction of a braided tensor category from an arbitrary tensor category.  Then, the center of a category appears as a tool to study categorical groups \cite{Car-Gar} and graded categorical groups \cite{Gar-Del}.

The detail proofs of the construction of $(\mathcal C, F)^\ast$ have showed in \cite{Maj2}. Concurrently, in \cite{Maj2}, S. Majid enriched the results of the dual categories and established links between dual categories and braided groups.

Monoidal categories were considered in a more general situation due to M. Laplaza with the name {\it distributivity category} \cite{Lap}. After, A. Fr\"{o}hllich and C. T. C.  Wall \cite{Fro-Wa} presented  the concept of {\it ring-like category}. These two concepts  are categorifications of the concept of commutative rings, as well as a generalization of the category of modules over a commutative ring $R$. The overlap of these two concepts has been proved in \cite{Q4}.

In order to have descriptions of structures, and a cohomological classification, N. T. Quang \cite{Q1} has introduced  the concept of {\it Ann-categories}, as a categorification of the concept of rings,  with requirements of invertibility of objects and morphisms of the under-lying category, similar to those of categorical groups (see \cite{Car-Gar}, \cite{Gar-Del}). In \cite{Q3}, N. T. Quang proved that  each congruence class of an Ann-category $\mathcal A$ is completely defined by three invariants: the ring $\Pi_0(\mathcal A)$ of congruence classes of objects of $\mathcal A,$ the $\Pi_0(\mathcal A)$-bimodule $\Pi_1(\mathcal A)$ of automorphisms of {\it zero} object, and an element in the cohomology group  $H^{3}_{MacL}(R, M)$ due to Mac Lane \cite{Mac}. The concept of {\it braided Ann-categories} is considered in \cite{Q4}, in which authors built the {\it center} of an Ann-category, an extension of the center construction of a monoidal category presented by A. Joyal and R. Street \cite{JS1}. This motivation leads to the purpose of this paper is to construct a dual Ann-category of an arbitrary Ann-category (in Section 3). This gives us a new framework of the concept of Ann-categories, which is very close to the ring extension problem. We also note that the center of an Ann-category is a dual over $\mathcal A$. Thus, in the duals over $\mathcal A$ there always exist braided Ann-categories.

In this paper, we sometimes denote by  $XY$ the tensor product of two objects  $X, Y$ instead of  $X\otimes Y$.


\section {Some basic definitions}

\begin{Definition}[\cite{Q1}]
An \emph{Ann-category} consists of:\\
(i)  Category $\mathcal A$ together with two bifunctors $\oplus , \otimes : \mathcal A \times \mathcal A \rightarrow \mathcal A$.\\
(ii) A fixed object $O\in \mathcal A$ together with naturality constraints $a^+,c^+,g,d$ such that $(\mathcal A,\oplus,a^+,c^+,(O,g,d))$ is a symmetric categorical group.\\
(iii) A fixed object $I\in\mathcal A$ together with naturality constraints $a,l,r$ such that $(\mathcal A,\otimes,a,(I,l,r))$ is a monoidal $A$-category.\\
(iv) Natural isomorphisms $\frak L, \frak R$
\begin{equation*}
\begin{array}{cccc}
\frak L_{A, X, Y} : &A\otimes (X\oplus Y) &\rightarrow& (A\otimes X)\oplus (A\otimes Y),\\
\frak R_{X, Y, A}:&(X\oplus Y)\otimes A &\rightarrow &(X\otimes A)\oplus (Y\otimes A),
\end{array}
\end{equation*}
such that the following conditions are satisfied:\\
(Ann-1) For each $A\in \mathcal A,$ the pairs $(L^A,\breve{L^A}),(R^A,\breve{R^A})$ defined by relations:
\begin{equation*}
\begin{array}{cccc}
L^A = A \otimes -,&&R^A = - \otimes A,\\
\Breve L_{X, Y}^A = \frak L_{A, X, Y},&&\Breve{R}_{X, Y}^A = \frak R_{X, Y, A}
\end{array}
\end{equation*}
are $\oplus$-functors which are compatible with $a^+$ and $c^+$.\\
(Ann-2) The following diagrams commute for   all  objects $ A,B,X,Y\in \mathcal A$:
\[\scriptsize\begin{diagram}
\node{(AB)(X\oplus Y)}\arrow{s,l}{\Breve L^{AB}}
\node{A(B(X\oplus Y))}\arrow{e,t}{id_A \otimes \Breve L^B}\arrow{w,t}{\ a_{A, B, X\oplus Y}}
\node{A(BX\oplus BY)}\arrow{s,r}{\Breve L^A}\\
\node{(AB)X\oplus (AB)Y}
\node[2]{A(BX)\oplus A(BY)}\arrow[2]{w,t}{a_{A, B, X}\oplus a_{A, B, Y}}
\end{diagram}\]
\[\scriptsize\begin{diagram}
\node{(X\oplus Y)(BA)}\arrow{s,l}{\Breve{R}^{BA}}\arrow{e,t}{a_{X\oplus Y, B, A}}
\node{((X\oplus Y)B)A}\arrow{e,t}{\Breve{R}^B \otimes id_A}
\node{(XB\oplus YB)A}\arrow{s,r}{\Breve{R}^A}\\
\node{X(BA)\oplus Y(BA)}\arrow[2]{e,t}{a_{X, B, A}\oplus a_{Y, B, A}}
\node[2]{(XB)A\oplus (YB)A}
\end{diagram}\]
\[\scriptsize\begin{diagram}
\node{(A(X\oplus Y))B}\arrow{s,l}{\Breve L^A \otimes id_B}
\node{A((X\oplus Y)B)}\arrow{w,t}{a_{A, X\oplus Y, B}}\arrow{e,t}{id_A \otimes \Breve{R}^B}
\node{A(XB\oplus YB)}\arrow{s,r}{\Breve L^A}\\
\node{(AX\oplus AY)B}\arrow{e,t}{\Breve{R}^B}
\node{(AX)B\oplus (AY)B}
\node{A(XB)\oplus A(YB)}\arrow{w,t}{a\  \oplus \ a}
\end{diagram}\]

\[\scriptsize
\begin{diagram}
\node{(A\oplus B)X\oplus (A\oplus B)Y}\arrow{s,l}{\Breve{R}^X \oplus \Breve{R}^Y}
\node{(A\oplus B)(X\oplus Y)}\arrow{w,t}{\Breve L}\arrow{e,t}{\Breve{R}}
\node{A(X\oplus Y)\oplus B(X\oplus Y)}\arrow{s,r}{\Breve L^A \oplus \Breve L^B}\\
\node{(AX\oplus BX)\oplus (AY\oplus BY)}\arrow[2]{e,t}{\text{v}}
\node[2]{(AX\oplus AY)\oplus (BX\oplus BY)}
\end{diagram}\]
where $ v = v_{_{U, V, Z, T}}:  (U\oplus V)\oplus (Z\oplus T) \rightarrow (U\oplus Z)\oplus (V\oplus T)$  
is the unique morphism built from $a^+,c^+,id$ in the symmetric monoidal category $(\mathcal A,\oplus).$

\noindent (Ann-3) For the unit object $I\in \mathcal A$ of the operation $\otimes,$ we have the following relations for all objects $X,Y\in \mathcal A$:

\begin{equation*}
l_{X\oplus Y}=(l_X\oplus l_Y)\circ \Breve L^I_{X,Y},\quad r_{X\oplus Y}=(r_X\oplus r_Y)\circ \Breve R^I_{X,Y}.
\end{equation*}
\end{Definition}

\begin{Definition} Let $\mathcal A$ and $\mathcal  A'$ be Ann-categories. An \emph{Ann-functor} from  
$\mathcal A$ to $\mathcal A'$ is a triple $(F, \Breve{F}, \widetilde{F})$, where $(F, \Breve{F})$ is a symmetric monoidal functor respect to the  operation $\oplus$, $(F, \widetilde{F})$ is an $A$-functor (i.e. an associativity functor) respect to the operation  $\otimes$, satisfying  the two following commutative diagrams for all $X, Y, Z \in Ob(\mathcal A)$:

\[\scriptsize\begin{diagram}
\node{F(X(Y\oplus Z))}\arrow{s,l}{F(\frak{L} )}
\node{FX.F(Y\oplus Z)}\arrow{w,t}{\widetilde{F}}
\node{FX(FY\oplus FZ)}\arrow{w,t}{id\otimes \Breve{F}}\arrow{s,r}{\frak{L} '}\\
\node{F(XY\oplus XZ)}
\node{F(XY)\oplus F(XZ)}\arrow{w,t}{\Breve{F}}
\node{FX.FY\oplus FX.FZ}\arrow{w,t}{\widetilde{F} \oplus \widetilde{F}}
\end{diagram}\]
\[\scriptsize\begin{diagram}
\node{F((X\oplus Y)Z)}\arrow{s,l}{F(\frak{R} )}
\node{F(X\oplus Y).FZ}\arrow{w,t}{\widetilde{F}} 
\node{(FX\oplus FY).FZ}\arrow{w,t}{\Breve{F} \otimes id}\arrow{s,r}{\frak{R} '}\\
\node{F(XZ\oplus YZ)}
\node{F(XZ)\oplus F(YZ)}\arrow{w,t}{\Breve{F}}
\node{FX.FZ\oplus FY.FZ}\arrow{w,t}{\widetilde{F} \oplus \widetilde{F}}
\end{diagram}\]
\end{Definition}

 \begin{Definition} A \emph{braided Ann-category}  $\mathcal A$ is an Ann-category $\mathcal A$ together with a  braid $c$ such that $(\mathcal A, \otimes, a, c, (I,l,r))$ is a braided tensor category, concurrently $c$ satisfies the following relation: 
 \begin{equation*}
(c_{A,X}\oplus c_{A,Y})\circ \Breve L^A_{X,Y}=\Breve R^A_{X,Y}\circ c_{A,X\oplus Y},
\end{equation*}
and the condition  $c_{_{O, O}}=id$.
\end{Definition}

Let us recall a result which has been known of an Ann-category.

\begin{Proposition}[{\cite[Proposition 3.1]{Q1}}]
In the Ann-category $\mathcal A$, there exist uniquely the isomorphisms:
$$\hat{L}^A:A\otimes O\rightarrow A,\qquad \qquad \hat{R}^A: O\otimes A\rightarrow A$$
\noindent such that $(L^A, \Breve{L}^A,\hat{L}^A),(R^A, \Breve{R}^A, \hat{R}^A)$ are the functors which are compatible with the unit constraints of the operator $\oplus$ (also called $U$-functors).
\end{Proposition}


\section{Duals of Ann-categories}

In this section, we shall build {\it duals of Ann-categories} based on the construction of duals of monoidal categories by S. Majid \cite{Maj}.

Let $\mathcal A$ be an Ann-category. An Ann-category $\mathcal B$ is {\it functored} over $\mathcal A$ if there is an Ann-functor $F:\mathcal B\rightarrow \mathcal A$.

First, let us recall that an Ann-category is called {\it almost strict} if all its natural constraints, except for the commutativity constraint and the left distributivity constraint, are identities. Each Ann-category is Ann-equivalent to an almost strict Ann-category of the type $(R,M)$ (see \cite{Q2}).  In this category, for each  $A\in Ob(\mathcal A)$, there exists an object  $A'\in Ob(\mathcal A)$ such that 
\begin{equation}\label{bd7}
A\oplus A'=O.\tag{1}
\end{equation}

So, hereafter, we always assume that $\mathcal A$ is an almost strict Ann-category and satisfies the condition (\ref{bd7}) and the Ann-functor $F:\mathcal B\rightarrow \mathcal A$ satisfies the conditions $F(O)=O, F(I)=I$.

\begin{Definition}
Let $\mathcal A$ be an Ann-category. Let $(\mathcal B, F)$ be a functored Ann-category over $\mathcal A$. A right $(\mathcal B, F)$-module is a pair $(A, u_A)$ consisting of an object $A$ in $\mathcal A$ and a natural transformation $u_{A,X}: A\otimes F(X)\rightarrow F(X)\otimes A$ such that  $u_{A,I}=id$  and the following diagrams commute:

\begin{equation}\label{bd9}
\scriptsize\begin{diagram}
\node{A\otimes (FX\oplus   FY)}\arrow{e,t}{\Breve L^A_{FX,FY}}\arrow{s,l}{id\otimes \Breve{F}}
\node{(A\otimes   FX)\oplus   (A\otimes   FY)}\arrow{e,t}{u_{A,X}\oplus   u_{A,Y}}
\node{(FX\otimes   A)\oplus  (FY\otimes   A)}\arrow{s,r}{id}\\
\node{A\otimes   F(X\oplus   Y)}\arrow{e,t}{u_{A,X\oplus   Y}}
\node{F(X\oplus   Y)\otimes   A}
\node{(FX\oplus   FY)\otimes   A}\arrow{w,t}{\Breve{F}\otimes   id}
\end{diagram}\tag{2}
\end{equation}

\begin{equation}\label{bd10}
\scriptsize\begin{diagram}
\node{A\otimes  (FX\otimes   FY)}\arrow{e,t}{u_{A,X}\otimes   id}\arrow{s,l}{id\otimes  \widetilde{F}}
\node{FX\otimes   A\otimes   FY}\arrow{e,t}{id\otimes   u_{A,Y}}
\node{FX\otimes   FY\otimes   A}\arrow{s,r}{\widetilde{F}\otimes   id}\\
\node{A\otimes   F(X\otimes   Y)}\arrow[2]{e,t}{u_{A,X\otimes   Y}}
\node[2]{F(X\otimes   Y)\otimes   A}
\end{diagram}\tag{3}
\end{equation}

\noindent A morphism $f: (A,u_{A})\rightarrow (B,u_{B})$ between right $(\mathcal B,F)$-modules is a morphism $f: A\rightarrow B$ in $\mathcal A$ such that the following diagram commutes for all $X\in \mathcal B$:
\begin{equation}\label{bd11}
\scriptsize\begin{diagram}
\node{A\otimes   FX}\arrow{e,t}{u_{A,X}}\arrow{s,l}{f\otimes   id}
\node{FX\otimes   A}\arrow{s,r}{id\otimes   f}\\
\node{B\otimes   FX}\arrow{e,t}{u_{B,X}}
\node{FX\otimes   B}
\end{diagram}
\tag{4}
\end{equation}
\end{Definition}

Let $(\mathcal B, F)$ be a  functored Ann-category over $\mathcal A.$ We consider the category $\mathcal B^{\ast}=(\mathcal B, F)^{\ast}$ defined as follows. The objects of $\mathcal B^{\ast}$ are right $(\mathcal B,F)$-modules. The morphisms of $\mathcal B^{\ast}$ are morphisms between right $(\mathcal B,F)$-modules.

Now, we shall equip the operators and the structures for $\mathcal B^{\ast}$ so that
$\mathcal B^{\ast}$ becomes an Ann-category.

\begin{Lemma}\label{thm1}
For any two objects $(A,u_{A}), (B,u_{B})$ in $\mathcal B^{\ast}$,  $(A\oplus B, u_{A\oplus B})$ is an object of $\mathcal B^{\ast}$, where $u_{A\oplus B}$ is defined by:
\begin{equation*}
u_{A\oplus B,X}=\frak L^{-1}_{FX,A,B}\circ (u_{A,X}\oplus u_{B,X}),\ \emph{for all}\  X\in \mathcal A.
\end{equation*}
\end{Lemma}

\begin{proof}
Since  $u_{A,I}=id,\  u_{B,I}=id,\ \frak L_{FI,A,B}=\frak L_{I,A,B}=id$, we have  $u_{A\oplus B,I}=id$.

To prove that $u_{A\oplus   B}$ satisfies the diagram (\ref{bd9}),  we consider the diagram (5) (see page 12). In the diagram (5), the regions (I), (VII) commute thanks to the determination of  $u_{A\oplus   B}$, the region (II) commutes thanks to the naturality of $\frak{R}=id$, the regions (III), (VI) commute since $\mathcal A$ is an  Ann-category, the region  (V) commutes thanks to the naturality of $\frak{L}$, the region (VIII) commutes thanks to the naturality of  $v$, the perimeter commutes since $(A,u_{A}), (B,u_{B})$ satisfy the diagram (\ref{bd9}). Therefore, the region (IV) commutes, i.e.,  $(A\oplus   B, u_{A\oplus   B})$ satisfies the diagram $(\ref{bd9})$.

To prove that  $u_{A\oplus   B}$ satisfies the diagram (\ref{bd10}), we consider the diagram (6) (see page 13). In the diagram (6), the regions (I), (II) commute thanks to the naturality of $\frak{R}=id$, the regions (III), (VI), (VIII) commute thanks to the determination of $u_{A\oplus   B}$, the regions (IV), (X) commute since $\mathcal A$ is an  Ann-category, the regions  (VII), (IX) commute thanks to the naturality of  $\frak{L}$, the perimeter commutes thanks to  $u_{A}, u_{B}$ satisfy the diagram (\ref{bd10}). Therefore, the region (V) commutes, i.e., $u_{A\oplus   B}$ satisfies the diagram (\ref{bd10}).
So, $(A\oplus   B, u_{A\oplus   B})$ is an object of $\mathcal B^{\ast}$.
\end{proof}
By Lemma \ref{thm1}, we can determine  the operator ``+" of $\mathcal B^{\ast}$ 
where the sum of two objects is defined by
\begin{equation*}
 (A,u_{A})+(B,u_{B})=(A\oplus B, u_{A\oplus B}),
\end{equation*}
and the sum of two morphisms  is the sum of morphisms in $\mathcal A$.

\begin{Proposition}\label{thm2}
$\mathcal B^{\ast}$ is a symmetric categorical group where
the associativity constraint  is strict, the unit constraint  is \  $((O, u_{O,X}=\hat{L}^{-1}_{FX}),id,id)$, and the commutativity constraint is \  $c^+_{(A,u_A),(B,u_B)}=c^+_{A,B}$.
\end{Proposition}

\begin{proof}
Assume that  $f: (A,u_{A})\rightarrow (B,u_{B})$ and $g: (C,u_{C})\rightarrow (D,u_{D})$ are two morphisms in the category $\mathcal B^{\ast}$. We shall prove that
\[f+g=f\oplus g\]
satisfies the diagram (\ref{bd11}), so it is a morphism of $\mathcal B^{\ast}$. We consider the diagram:
\[\scriptsize\setlength\unitlength{0.5cm}
\begin{picture}(40,11)
\put(6,10){$AFX \oplus   CFX$}
\put(6,7){$(A\oplus   C)FX$}
\put(6,4){$(B\oplus   D)FX$}
\put(6,1){$BFX \oplus   DFX$}

\put(17,10){$(FX)A \oplus   (FX)C$}
\put(17,7){$(FX)(A\oplus   C)$}
\put(17,4){$(FX)(B\oplus   D)$}
\put(17,1){$(FX)B \oplus   (FX)D$}

\put(7.5,9.8){\line(0,-1){2.2}}\put(7.4,9.8){\line(0,-1){2.2}}
\put(7.5,6.8){\vector(0,-1){2.2}}\put(7.7,5,5){$(f\oplus   g)\otimes   id$}
\put(7.5,3.8){\line(0,-1){2.2}}\put(7.4,3.8){\line(0,-1){2.2}}

\put(19,7.5){\vector(0,1){2.2}}\put(19.2,8.5){$\Breve{L}$}
\put(19,6.8){\vector(0,-1){2.2}}\put(15.3,5,5){$id\otimes   (f\oplus   g)$}
\put(19,3.8){\vector(0,-1){2.2}}\put(19.2,2.5){$\Breve{L}$}

\put(9.7,10.1){\vector(1,0){7}}\put(11.4,10.5){$u_{A,X}\oplus   u_{C,X}$}
\put(9.7,7.1){\vector(1,0){7}}\put(11.9,7.5){$u_{A\oplus   C,X}$}
\put(9.7,4.1){\vector(1,0){7}}\put(11.9,4.5){$u_{B\oplus   D,X}$}
\put(9.7,1.1){\vector(1,0){7}}\put(11.4,1.5){$u_{B,X}\oplus   u_{D,X}$}

\put(5.8,10.1){\line(-1,0){5}}
\put(0.8,10.1){\line(0,-1){9}}
\put(0.8,1.1){\vector(1,0){5}}
\put(1,5.5){$(f\otimes   id)\oplus (g\otimes   id)$}

\put(21.8,10.1){\line(1,0){5}}
\put(26.8,10.1){\line(0,-1){9}}
\put(26.8,1.1){\vector(-1,0){4.5}}
\put(21.7,5.5){$(id\otimes   f)\oplus (id\otimes   g)$}

\put(3.5,8.5){(I)}\put(12.5,8.5){(II)}
\put(12.45,5.5){(III)}
\put(12.5,2.5){(IV)}
\put(22,8.5){(V)}
\end{picture}\]

In this diagram, the region (I) commutes thanks to the naturality of $\frak{R}=id$, the region (II) commutes thanks to the determination of  $u_{A\oplus   C}$, the region (IV) commutes thanks to the determination of $u_{B\oplus   D}$, the region (V) commutes thanks to the naturality of  $\frak{L}$; each component of the perimeter commutes since $f$ and $g$ are morphisms of $\mathcal B^{\ast}$. So, the perimeter commutes. Therefore, the region (III) commutes, i.e.,  $f+g=f\oplus   g$ is a morphism of $\mathcal B^{\ast}$.

Next, we prove that $a^+=id$ is a morphism
\[((A,u_{A})+(B,u_{B}))+(C,u_{C})\rightarrow (A,u_{A})+((B,u_{B})+(C,u_{C}))\]
in $\mathcal B^{\ast}$. We consider the following diagram:
\[\scriptsize\setlength\unitlength{0.5cm}
\begin{picture}(26,17)
\put(4,16){$(AFX \oplus   BFX)\oplus   CFX$}
\put(4,13){$(A\oplus   B)FX \oplus   CFX$}
\put(4,10){$((A\oplus   B)\oplus   C)FX$}
\put(4,7){$(FX)((A\oplus   B)\oplus   C)$}
\put(4,4){$(FX)(A\oplus   B)\oplus   (FX)C$}
\put(4,1){$((FX)A\oplus   (FX)B)\oplus   (FX)C$}

\put(16,16){$AFX \oplus   (BFX\oplus   CFX)$}
\put(16,13){$AFX\oplus   (B\oplus   C)FX$}
\put(16,10){$(A\oplus   (B\oplus   C))FX$}
\put(16,7){$(FX)(A\oplus   (B\oplus   C))$}
\put(16,4){$(FX)A\oplus   (FX)(B\oplus   C)$}
\put(16,1){$(FX)A\oplus   ((FX)B\oplus   (FX)C)$}

\put(6.5,15.8){\line(0,-1){2.2}}\put(6.4,15.8){\line(0,-1){2.2}}
\put(6.5,12.8){\line(0,-1){2.2}}\put(6.4,12.8){\line(0,-1){2.2}}
\put(6.5,9.8){\vector(0,-1){2.2}}\put(6.7,8.5){$u_{(A\oplus   B)\oplus   C,X}$}
\put(6.5,6.8){\vector(0,-1){2.2}}\put(6.7,5.5){$\Breve{L}$}
\put(6.5,3.8){\vector(0,-1){2.2}}\put(6.7,2.5){$\Breve{L}\otimes   id$}

\put(18.5,15.8){\line(0,-1){2.2}}\put(18.4,15.8){\line(0,-1){2.2}}
\put(18.5,12.8){\line(0,-1){2.2}}\put(18.4,12.8){\line(0,-1){2.2}}
\put(18.5,9.8){\vector(0,-1){2.2}}\put(15,8.5){$u_{A\oplus   (B\oplus   C),X}$}
\put(18.5,6.8){\vector(0,-1){2.2}}\put(18.7,5.5){$\Breve{L}$}
\put(18.5,3.8){\vector(0,-1){2.2}}\put(18.7,2.5){$id\otimes   \Breve{L}$}

\put(10.4,16.1){\line(1,0){5.3}}\put(10.4,16.2){\line(1,0){5.3}}
\put(8.9,10.1){\line(1,0){6.8}}\put(8.9,10.2){\line(1,0){6.8}}
\put(9.5,7.1){\line(1,0){6.3}}\put(9.5,7.2){\line(1,0){6.3}}
\put(11.9,1.1){\line(1,0){3.8}}\put(11.9,1.2){\line(1,0){3.8}}

\put(3.8,13.1){\line(-1,0){2}}
\put(1.8,13.1){\line(0,-1){9}}
\put(1.8,4.1){\vector(1,0){2}}
\put(2.2,8.5){$\alpha_2  $}

\put(3.8,16.1){\line(-1,0){4}}
\put(-0.2,16.1){\line(0,-1){15}}
\put(-0.2,1.1){\vector(1,0){4}}
\put(0.1,8.5){$\alpha_1$}

\put(21.5,13.1){\line(1,0){1.5}}
\put(23,13.1){\line(0,-1){9}}
\put(23,4.1){\vector(-1,0){0.6}}
\put(21.7,8.5){$\alpha_3$}

\put(22.5,16.1){\line(1,0){2.5}}
\put(25,16.1){\line(0,-1){15}}
\put(25,1.1){\vector(-1,0){1.1}}
\put(25.5,8.5){$\alpha_4$}

\put(3,14.5){(I)}\put(3,11.5){(III)}
\put(21,14.5){(II)}\put(21,11.5){(IV)}
\put(12,8.5){(V)}
\put(12,4){(VI)}

\put(0, -1){{\emph  where} \ $\alpha_1=(u_{A, X}\oplus u_{B,X})\oplus u_{C,X}$}
\put(15,-1) {$\alpha_2=u_{A\oplus B, X}\oplus u_{C,X}$}
\put(1.8,-2) {$\alpha_3=u_{A,X}\oplus u_{B\oplus C,X}$}
\put(15,-2){$\alpha_4=u_{A, X}\oplus (u_{B,X}\oplus u_{C,X})$}
\end{picture}\]

\vspace{1cm}

In the above diagram, the region (I) commutes thanks to the determination of  $u_{A\oplus   B}$, the region (II) commutes thanks to the determination of $u_{B\oplus   C}$, the region (III) commutes thanks to the determination of $u_{(A\oplus   B)\oplus   C}$, the region (IV) commutes thanks to the determination of $u_{A\oplus   (B\oplus   C)}$, 
the region (VI) commutes since $\mathcal A$ is an Ann-category, the perimeter commutes thanks to the naturality of $a^+=id$. Therefore, the region (V) commutes, i.e., $a^+=id$ is a morphism of $\mathcal B^{\ast}$.

To prove that $c^+$ is the morphism
\[(A,u_{A})+(B,u_{B})\rightarrow (B,u_{B})+(A,u_{A})\]
in $\mathcal B^{\ast}$, we consider the following diagram. In this diagram, the region (I) commutes thanks to the determination of  $u_{A\oplus   B}$, the regions  (II),  (IV) commute since $\mathcal A$ is an Ann-category, the region (V) commutes thanks to the determination of $u_{B\oplus   A}$, the perimeter commutes thanks to the naturality of  $c^+$. Therefore, the region (III) commutes, i.e., $c^+$ is a morphism in $\mathcal B^{\ast}$.
\[\scriptsize\setlength\unitlength{0.5cm}
\begin{picture}(26,11)
\put(5,10){$AFX \oplus   BFX$}
\put(5,7){$(A\oplus   B)FX$}
\put(5,4){$(FX)(A\oplus   B)$}
\put(5,1){$(FX)A \oplus   (FX)B$}

\put(16,10){$BFX \oplus   AFX$}
\put(16,7){$(B\oplus   A)FX$}
\put(16,4){$(FX)(B\oplus   A)$}
\put(16,1){$(FX)B \oplus   (FX)A$}

\put(6.5,9.8){\line(0,-1){2.2}}\put(6.4,9.8){\line(0,-1){2.2}}
\put(6.5,6.8){\vector(0,-1){2.2}}\put(6.7,5,5){$u_{A\oplus   B,X}$}
\put(6.5,3.8){\vector(0,-1){2.2}}\put(6.7,2.5){$\Breve{L}$}

\put(17.6,9.8){\line(0,-1){2.2}}\put(17.7,9.8){\line(0,-1){2.2}}

\put(17.6,6.8){\vector(0,-1){2.2}}\put(17.8,5,5){$u_{B\oplus   A,X}$}
\put(17.6,3.8){\vector(0,-1){2.2}}\put(17.8,2.5){$\Breve{L}$}

\put(8.7,10.1){\vector(1,0){7}}\put(12,10.5){$c^+$}
\put(8.7,7.1){\vector(1,0){7}}\put(11.5,7.5){$c^+\otimes   id$}
\put(8.7,4.1){\vector(1,0){7}}\put(11.5,4.5){$id\otimes   c^+$}
\put(9.6,1.1){\vector(1,0){6.1}}\put(12,1.5){$c^+$}

\put(4.8,10.1){\line(-1,0){4}}
\put(0.8,10.1){\line(0,-1){9}}
\put(0.8,1.1){\vector(1,0){4}}
\put(1.1,5.5){$u_{A,X}\oplus u_{B,X} $}

\put(20,10.1){\line(1,0){2.5}}
\put(22.5,10.1){\line(0,-1){9}}
\put(22.5,1.1){\vector(-1,0){1.7}}
\put(22.7,5.7){$u_{B,X}\oplus u_{A,X} $}

\put(3,8.5){(I)}\put(11.5,8.5){(II)}
\put(11.45,5.5){(III)}
\put(11.5,2.5){(IV)}
\put(21,5.5){(V)}
\end{picture}\]
One can verify that
$((O, u_{O, X}=\hat L^{-1}_{FX}),id,id)$  is the unit constraint of $\mathcal B^{\ast}$. Finally, we shall prove that each object of  $\mathcal B^{\ast}$ is invertible.

Let  $(A,u_A)$ be an object of $\mathcal B^{\ast}$. By the condition (\ref{bd7}), there exsits  an object $A'\in Ob(\mathcal A)$ such that 
\[A\oplus   A'=O.\]
The family of natural transformations $u_{A', X}: A'\otimes   FX\rightarrow FX\otimes   A'$ is defined by:
\begin{equation*}
u_{A, X}\oplus   u_{A', X}=\frak{L}_{FX, A, A'}\circ u_{O, X}.
\end{equation*}

One can prove that $(A', u_{A'})$ is the invertible object of the object $(A,u_A)$ in the category $\mathcal B^\ast$.
\end{proof}
\begin{Lemma}\label{thm3}
For any two objects $(A,u_{A}), (B,u_{B})$ of $\mathcal B^{\ast}$, $(A\otimes B, u_{A\otimes B})$ is an object of $\mathcal B^{\ast}$, where $u_{A\otimes B}$ is defined by:
\begin{equation*}
u_{A\otimes B,X}=(u_{A,X}\otimes id_B)\circ (id_A\otimes u_{B,X}), \ \emph{for all}\   X\in \mathcal A.
\end{equation*}
\end{Lemma}

\begin{proof}
Let $(A,u_{A}), (B,u_{B})$ be two objects of $\mathcal B^{\ast}$. Since\  $u_{A,I}=id$ and  $u_{B,I}=id$, we have $u_{A\otimes B, I}=id$. Moreover, by Theorem 3.3 \cite{Maj}, $u_{A\otimes B}$ satisfies the diagram  $(\ref{bd10})$. 

Finally, to prove that $u_{A\otimes B}$ satisfies the diagram (\ref{bd9}), we consider the diagram (7) (see page 14). In the  diagram (7), the region (I) commutes since  $(B,u_{B})$ satisfies the diagram (\ref{bd9}), the regions (II), (VII) and (IX) commute thanks to the naturality of $a^+=id$, the region (III) commutes thanks to the naturality of $\frak{L}$, the regions (IV), (XI) and the perimeter commutes since $\mathcal A$ is an Ann-category,  the regions (VI), (VIII) commute thanks to the determination of $u_{AB}$, the region (X) commutes since $(A,u_{A})$ satisfies the diagram (\ref{bd9}), the region (XII) commutes thanks to the naturality of $\frak{R}=id$. Therefore, the region (V) commutes, i.e., $(AB,u_{AB})$ satisfies the diagram (\ref{bd9}). So $(A\otimes B, u_{A\otimes B})$ is an object of $\mathcal B^{\ast}$.
\end{proof}

By Lemma \ref{thm3}, we can determine  the operator ``$\times$" of $\mathcal B^{\ast}$ 
where the product of two objects is defined by
\begin{equation*}
(A,u_{A})\times (B,u_{B})=(A\otimes B, u_{A\otimes B}),
\end{equation*}
and the tensor product of two morphisms is  the tensor product of two morphisms in $\mathcal A$.

\begin{Proposition}\label{thm4}
$\mathcal B^{\ast}$ is a strict monoidal category.
\end{Proposition}

\begin{proof}
Assume that $f: (A,u_{A})\rightarrow (B,u_{B})$ and $g: (C,u_{C})\rightarrow (D,u_{D})$ are two morphisms in the category $\mathcal B^{\ast}$. By Theorem 3.3 \cite{Maj}, the morphism
\begin{equation*}
f\times g=f\otimes g: (A,u_{A})\times (C,u_{C})\rightarrow (B,u_{B})\times (D,u_{D})
\end{equation*}
satisfies the diagram $(\ref{bd11})$, i.e., $f\times g$ is a morphism in $\mathcal B^{\ast}$. 

The composition of two morphisms  in $\mathcal B^{\ast}$ is the normal composition.
By Theorem 3.3 \cite{Maj}, $\mathcal B^{\ast}$ has the associativity constraint be strict.
One can easily prove that  $(I, id)$ is an object in $\mathcal B^{\ast}$ and it together with the strict constraints $l=id, r=id$ is the unit constraint of the operator $\times$  in $\mathcal B^{\ast}$.
\end{proof}

\begin{Theorem}\label{thm5}
$\mathcal B^{\ast}$ is an Ann-category with the distributivity constraints are given by
\begin{eqnarray*}
\frak L_{(A,u_{A}), (B,u_{B}), (C,u_{C})}=\frak L_{A,B,C}, \  
\frak R_{(A,u_{A}), (B,u_{B}), (C,u_{C})}=id.
\end{eqnarray*}
 \end{Theorem}
\begin{proof}
By Proposition \ref{thm2}, $(\mathcal B^{\ast}, +)$ is a symmetric categorical group. By Proposition \ref{thm4}, $(\mathcal B^{\ast}, \times)$ is a monoidal category. 
One can  prove that
\begin{eqnarray*} 
\frak{L}:(A,u_{A})\times ((B,u_{B})+(C,u_{C}))\rightarrow (A,u_{A})\times (B,u_{B})+ (A,u_{A})\times (C,u_{C}),\\
\frak{R}=id:((A,u_{A} + (B,u_{B})) \times (C,u_{C})\rightarrow (A,u_{A})\times (C,u_{C})+ (B,u_{B})\times (C,u_{C})
\end{eqnarray*}
are morphisms in $\mathcal B^{\ast}$.

Moreover, the constraints  $a^+=id, c^+, a=id, \frak{L}, \frak{R}=id$ of the  Ann-category $\mathcal A$ satisfy the conditions (Ann-1), (Ann-2), (Ann-3), so, in the category $\mathcal B^{\ast}$, they also satisfy these conditions. Thus $\mathcal B^{\ast}$ is an  Ann-category.
\end{proof}

\newpage
The following proposition is obvious.

\begin{Proposition} $\mathcal B^{\ast}$ is functored over $\mathcal A$ with the forgetful Ann-functor\linebreak
$$F^{\ast}:\mathcal B^{\ast}\rightarrow \mathcal A.$$
 \end{Proposition}

\noindent{\bf Example 1. The center of an Ann-category $\mathcal A$}

Let  $\mathcal A$ be an Ann-category. Let $\mathcal B=\mathcal A$ and $F=id$. Then $\mathcal B^{\ast}=\mathcal C_{\mathcal A}$, where $\mathcal C_{\mathcal A}$ is the center of the Ann-category $\mathcal A$ which is built in \cite{Q4}. This is a braided Ann-category with the quasi-symmetric
 \[c_{(A,u_A), (B,u_B)}=u_{A,B}: A\otimes   B\rightarrow B\otimes   A.\]

Next, we shall apply above results to build the  dual Ann-category of the pair $(\mathcal B,F)$, where $\mathcal B=(R',M',f')$, $\mathcal A=(R,M,f)$ are Ann-categories.
\vspace{0.2cm}

\noindent{\bf  Example 2.\ Duals of an Ann-category of the type $(R,M)$}

Let $R$ be a ring and   $M$ be a  $R$-bimodule.  An  Ann-category  of the type $(R,M)$ is a category $\mathcal I$ whose objects are elements of $R$, and whose morphisms are automorphisms, $(x,a): x\rightarrow x, \ \forall a\in M$. 
The composition of morphisms is the addition in $M$. The two operators  $\oplus$ and $\otimes$ of  $\mathcal I$ are given by
\begin{eqnarray*}
x\oplus y=x+y, & (x,a)\oplus (y,b)=(x+y,a+b),\\
x\otimes y=x.y,& (x,a)\otimes (y,b)=(xy, xb+ay).
\end{eqnarray*}

All constraints of $\mathcal I$ are strict, except for the left distributivity constraint and the commutativity constraint given by
\begin{eqnarray*}
\frak{L}_{x,y,z}&=&(\bullet, \lambda(x,y,z)):x(y+z)\rightarrow xy+xz,\\ 
c^+_{x,y}&=&(\bullet,\eta(x,y)): x+y\rightarrow y+x,
\end{eqnarray*}
where $\lambda:R^{3}\rightarrow M, \eta: R^{2}\rightarrow M $  are functions satisfying the some certain coherence conditions (for detail, see \cite{Q2}, \cite{Q3}).

Let $\mathcal A$ be an almost strict  Ann-category of the type $(R,M)$ and $\mathcal B$  be an
almost strict  Ann-category of the type  $(R',M')$. Let $(F, \Breve{F},\widetilde{F}): \mathcal B\rightarrow \mathcal A$ be an Ann-functor. Then, by Theorem 4.3 \cite{Q5}, $F$ is a functor of the type $(p,q)$, i.e.,
\[F(x) = p(x), F(x, a) = (p(x), q(a)),\]
where  $p: R'\rightarrow R $ is a ring homomorphism and $q: M' \rightarrow  M$ is a group homomorphism and 
\[q(xa) = p(x)q(a),\  \  q(ax) = q(a)p(x),\ \textrm{for all}\ x \in R, a \in M. \]
 Moreover, $\Breve{F},\widetilde{F}$ are associated, respectively, to $\mu, \nu$ which satisfy some certain coherence conditions (for detail, see Theorem 4.4 \cite{Q5}).

According to the above steps, each object of   $\mathcal B^{\ast}$ is a pair $(r,u_{r})$, where $r$ is in the centerization of $Imp=p(R')$ in the ring $R$, (i.e., $rp(x)=p(x)r\ \forall x\in R'$)
and $u_{r}:R'\rightarrow M$ is a function satisfying the condition  $u_{r,1}=0$ and the two following conditions for all $x,y\in R'$:
 \begin{eqnarray*}
u(r,x)-u(r,x+y)+u(r,y)&=&\mu(x,y)r+r\mu(x,y)-\lambda(r,px,py),\\
xu(r,y)-u(r,xy)+u(r,x)y&=&r\nu(x,y)-\nu(x,y)r.
\end{eqnarray*}
We now describe a morphism
$f: (r,u_{r})\rightarrow (s,u_{s})$  of $\mathcal B^{\ast}$. Since $f: r\rightarrow s$ is a morphism in the Ann-category $\mathcal A$, $s=r$, and $f=(r,a)$ with $a\in M$.

From the commutation of the diagram (\ref{bd11}), we have
\begin{equation*}
p(x)a=ap(x),\ \textrm{for all}\ x\in R'.
\end{equation*}

Now,  $\mathcal B^{\ast}$  is an Ann-category with the two operators given by
\begin{eqnarray*}
(r,u_{r})+(s,u_{s})&=&(r+s,u_{r+s}),\\
(r,u_{r})\times (s,u_{s})&=&(rs,u_{rs}),
\end{eqnarray*}
where
\begin{eqnarray*}
u_{r+s,x}&=&u_{r,x}+u_{s,x}-\lambda(px,r,s),\\
u_{rs,x}&=&u_{r,x}s+r.u_{s,x},
\end{eqnarray*}
and $f+g=f\oplus g$,\ $f\times g=f\otimes g$ where $f: (r,u_{r})\rightarrow (r,u_{r}), g:(s,u_{s})\rightarrow (s,u_{s})$.

All constraints of $\mathcal B^{\ast}$ are strict, except for the commutativity constraint and the left distributivity constraint given by
\begin{eqnarray*}
c^{+}_{(r,u_{r}), (s,u_{s})}&=& c^+_{r,s}=(\bullet,\eta(r,s)),\\
\frak L_{(r,u_{r}), (s,u_{s}), (t,u_{t})}&=&\frak L_{r,s,t}=(\bullet, \lambda(r,s,t)).
\end{eqnarray*}

The invertible object of the object $(r,u_{r})$ respect to the operator $+$ is  $(-r,u_{-r})$, where $-r$ is the opposite element of  $r$ in the group $(R,+)$ and $u_{-r}: R'\rightarrow M$ is given by:

\begin{equation*}
u_{-r,x}=\lambda(px,r,-r)-u_{r,x}.
\end{equation*}

\newpage
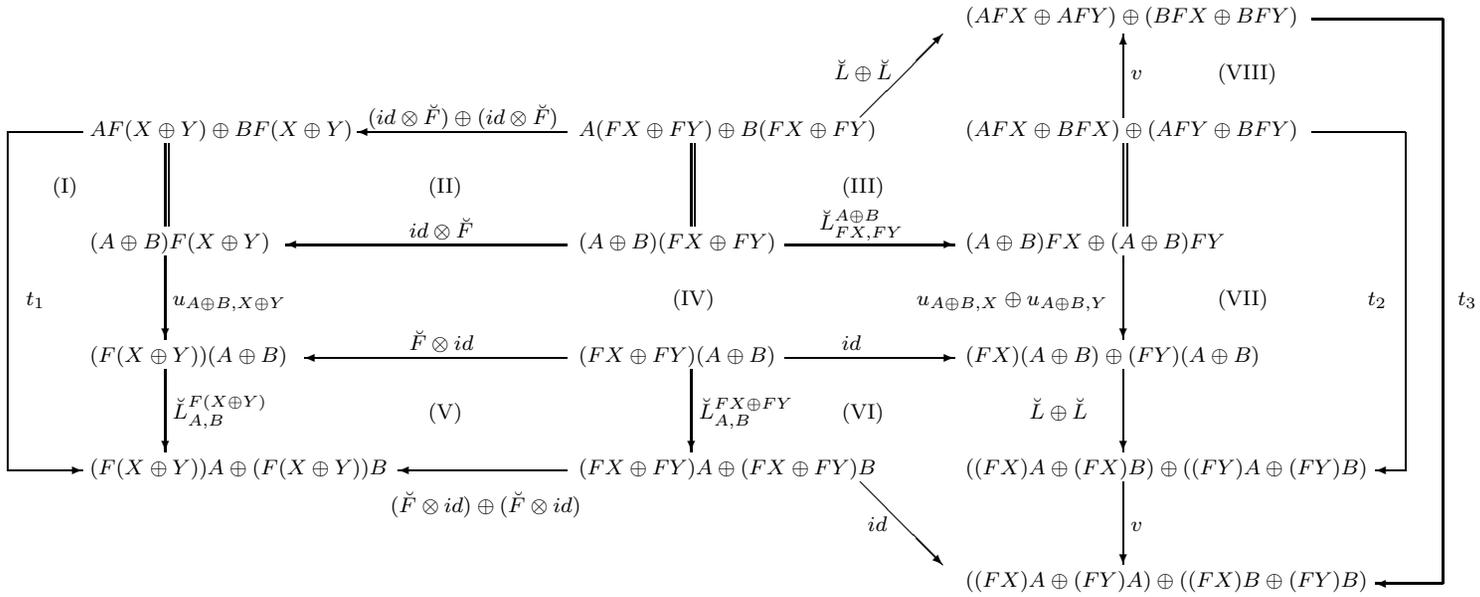
\[\begin{sideways}
\scriptsize\setlength\unitlength{0.5cm}
\begin{picture}(40,16)

\put(2, 13){$A F(X\oplus   Y)\oplus B F(X\oplus   Y)$}
\put(2, 10){$(A\oplus   B) F(X\oplus   Y)$}
\put(2, 7){$(F(X \oplus   Y))(A\oplus   B)$}
\put(2, 4){$(F(X\oplus    Y))A\oplus   (F(X\oplus   Y))B$}

\put(15,13){$A(FX\oplus   FY)\oplus   B(FX\oplus   FY)$}\put(22.5,13.5){\vector(1,1){2.2}}
                                        \put(21.8,14.5){$\Breve{L}\oplus   \Breve{L}$}
\put(15,10){$(A\oplus   B)(FX\oplus   FY)$}\put(20.5,10.1){\vector(1,0){4.5}}
                                 \put(21.4,10.5){$\Breve{L}^{A\oplus     B}_{FX,FY}$}
\put(15,7){$(FX\oplus   FY)(A\oplus   B)$}\put(20.5,7.1){\vector(1,0){4.5}}
                                 \put(22,7.3){$id$}
\put(15,4){$(FX\oplus   FY)A\oplus   (FX\oplus   FY)B$}\put(22.5,3.8){\vector(1,-1){2.2}}
                                        \put(22.7,2.5){$id$}

\put(25.3,16){$(AFX\oplus   AFY)\oplus   (BFX\oplus   BFY)$}
\put(25.3,13){$(AFX\oplus   BFX)\oplus   (AFY\oplus   BFY)$}
\put(25.3,10){$(A\oplus   B)FX\oplus   (A\oplus   B)FY$}
\put(25.3,7){$(FX)(A\oplus   B)\oplus   (FY)(A\oplus   B)$}
\put(25.3,4){$((FX)A\oplus   (FX)B)\oplus   ((FY)A\oplus   (FY)B)$}
\put(25.3,1){$((FX)A\oplus   (FY)A)\oplus   ((FX)B\oplus   (FY)B)$}

\put(4,12.8){\line(0,-1){2.2}}
\put(4.1,12.8){\line(0,-1){2.2}}
\put(4,9.8){\vector(0,-1){2.2}}\put(4.2,8.5){$u_{A\oplus   B, X\oplus   Y}$}
\put(4,6.8){\vector(0,-1){2.2}}\put(4.2,5.5){$\Breve{L}^{F(X\oplus   Y)}_{A,B}$}

\put(18,12.8){\line(0,-1){2.2}}
\put(18.1,12.8){\line(0,-1){2.2}}
\put(18,6.8){\vector(0,-1){2.2}}\put(18.2,5.5){$\Breve{L}^{FX\oplus   FY}_{A,B}$}

\put(14.7,4.1){\vector(-1,0){4.5}}\put(10,3){$(\Breve{F}\otimes   id) \oplus   (\Breve{F}\otimes   id)$}
\put(14.7,7.1){\vector(-1,0){7}}\put(10.5,7.3){$\Breve{F}\otimes   id$}
\put(14.7,10.1){\vector(-1,0){7.5}}\put(10.5,10.3){$id \otimes   \Breve{F}$}
\put(14.7,13.1){\vector(-1,0){5.6}}\put(9.4,13.3){$(id\otimes   \Breve{F}) \oplus   (id\otimes   \Breve{F})$}

\put(29.5,13.5){\vector(0,1){2.2}}\put(29.7,14.5){$v$}
\put(29.5,12.8){\line(0,-1){2.2}}\put(29.6,12.8){\line(0,-1){2.2}}
\put(29.5,9.8){\vector(0,-1){2.2}}\put(24,8.5){$u_{A\oplus   B, X}\oplus   u_{A\oplus   B, Y}$}
\put(29.5,6.8){\vector(0,-1){2.2}}\put(27,5.5){$\Breve{L}\oplus   \Breve{L}$}
\put(29.5,3.8){\vector(0,-1){2.2}}\put(29.7,2.5){$v$}

\put(1.8,13.1){\line(-1,0){2}}
\put(-0.2,13.1){\line(0,-1){9}}
\put(-0.2,4.1){\vector(1,0){2}}
\put(0.3,8.5){$t_1$}

\put(34.5,13.1){\line(1,0){2.5}}
\put(37,13.1){\line(0,-1){9}}
\put(37,4.1){\vector(-1,0){0.8}}


\put(36, 8.5){$t_2$}

\put(34.5,16.1){\line(1,0){3.5}}
\put(38,16.1){\line(0,-1){15}}
\put(38,1.1){\vector(-1,0){1.8}}

\put(38.4, 8.5){$t_3$}


\put(1,11.5){(I)}\put(11,11.5){(II)}\put(22,11.5){(III)}
                 \put(11,5.5){(V)}\put(22,5.5){(VI)}
                 
\put(17.5,8.5){(IV)}
\put(32,8.5){(VII)}
\put(32,14.5){(VIII)}
\put(17,0){\text{Diagram (5)}}

\put(0, -1){{\emph where} \ $t_1=u_{A, X\oplus Y}\oplus u_{B, X\oplus Y}$}
\put(1.8,-2) {$t_2=(u_{A,X}\oplus u_{B,X})\oplus (u_{A,Y}\oplus u_{B,Y})$}
\put(1.8,-3) {$t_3=(u_{A,X}\oplus u_{A,Y})\oplus (u_{B,X}\oplus u_{B,Y})$}
\end{picture}
\end{sideways}\]

\newpage
\[\begin{sideways}
\scriptsize\setlength\unitlength{0.5cm}
\begin{picture}(42,16)
\put(2, 16){$A (FX FY)\oplus   B (FX FY)$}\put(1.8,16.1){\line(0,-1){9}}\put(1,8.5){$id$}
                                                                      \put(1.8,7.1){\vector(1,0){0.8}} 

\put(2, 13){$(A FX) FY\oplus   (B FX) FY$}\put(8.4,13.1){\vector(1,0){7.3}}
                               \put(8.5,13.5){$(u_{A,X}\otimes   id)\oplus   (u_{B,X}\otimes   id)$}
                               
\put(2, 10){$(A FX\oplus   B FX) FY$}\put(7,10.1){\vector(1,0){4.8}}
                                \put(6.9,10.5){$(u_{A,X}\oplus   u_{B,X})\otimes   id$}

\put(2.5, 8.5){$((A \oplus   B) FX) FY$}

\put(3, 7){$(A \oplus   B) (FX FY)$}\put(8,7.1){\vector(1,0){8}}
                               \put(9.5,7.5){$u_{A\oplus   B,X}\otimes   id$}
                               
\put(2, 4){$(A \oplus   B) F(X Y)$}\put(6.2,4.1){\vector(1,0){24}}
                               \put(17,4.5){$u_{A\oplus   B,XY}$}
\put(2, 1){$A F(X Y) \oplus   BF(XY)$}\put(7.7,1.1){\vector(1,0){22.1}}
                               \put(16,1.5){$u_{A,XY}\oplus   u_{B,XY}$}

\put(12, 10){$((FX)A \oplus   (FX)B) FY$}\put(17.1,10.5){\vector(1,1){2.2}}\put(17.3,11.5){$id$}

\put(19.7, 10){$(FX)(AFY \oplus   BFY)$}\put(25.2,10.1){\vector(1,0){4.8}}
                                \put(25,10.5){$id\otimes   (u_{A,Y}\oplus   u_{B,Y})$}
                                
\put(21.8,10.5){\vector(-1,1){2.2}}\put(21,11.5){$\Breve{L}$}

\put(16, 13){$((FX)A)FY \oplus   ((FX)B)FY$}\put(23.4,13.1){\vector(1,0){6.6}}
                               \put(23.4,13.5){$(id\otimes   u_{A,Y})\oplus   (id\otimes   u_{B,Y})$}
                               
\put(17, 7){$(FX)(A \oplus   B)FY$}\put(21.7,7.1){\vector(1,0){8}}
                               \put(25,7.5){$id\otimes   u_{A\oplus   B,Y}$}
                               
                               \put(19,7.5){\vector(-1,1){2.2}}\put(15.7,8.5){$\Breve{L}\otimes   id$}
                               \put(19.5,7.5){\vector(1,1){2.2}}\put(21,8.5){$id$}

\put(30.3, 16){$(FX FY)A\oplus    (FX FY)B$}
\put(30.3, 13){$(FX) ((FY)A)\oplus   (FX) ((FY)B)$}
\put(30.3, 10){$(FX)((FY)A\oplus   (FY)B)$}
\put(30.3, 7){$(FX)((FY)(A \oplus   B))$}
\put(32, 5.7){$(FX FY)(A \oplus   B)$}

\put(30.3, 4){$F(XY)(A \oplus   B)$}
\put(30.3, 1){$(F(X Y))A \oplus   (F(XY))B$}

\put(5,15.8){\line(0,-1){2.2}}
\put(4.9,15.8){\line(0,-1){2.2}}

\put(5,12.8){\line(0,-1){2.2}}
\put(4.9,12.8){\line(0,-1){2.2}}

\put(5,9.8){\line(0,-1){0.7}}
\put(4.9,9.8){\line(0,-1){0.7}}

\put(5,8.4){\line(0,-1){0.7}}
\put(4.9,8.4){\line(0,-1){0.7}}

\put(5, 6.8){\vector(0,-1){2.2}}\put(3.2, 5.5){$id \otimes   \widetilde{F}$}

\put(5,3.8){\line(0,-1){2.2}}
\put(4.9,3.8){\line(0,-1){2.2}}


\put(33,15.8){\line(0,-1){2.2}}
\put(32.9,15.8){\line(0,-1){2.2}}

\put(33,10.5){\vector(0,1){2.2}}\put(29,11.5){$\Breve{L}^{FX}_{(FY)A,(FY)B}$}

\put(33,7.5){\vector(0,1){2.2}}\put(30,8.5){$id\otimes   \Breve{L}^{FY}_{A,B}$}

\put(33, 6.8){\line(0,-1){0.5}}\put(32.9, 6.8){\line(0,-1){0.5}}

\put(33,5.5){\vector(0,-1){1}}\put(33.2,5){$\widetilde{F}\otimes   id$}

\put(33,3.8){\vector(0,-1){2.2}}\put(33.2,2.5){$\Breve{L}^{F(XY)}_{A,B}$}

\put(1.8,16.1){\line(-1,0){2}}
\put(-0.2,16.1){\line(0,-1){15}}
\put(-0.2,1.1){\vector(1,0){2}}

\put(0.2, 5.5){$t_4$}


\put(37,16.3){\line(1,0){4}}
\put(41,16.3){\line(0,-1){15.2}}
\put(41,1.1){\vector(-1,0){3.5}}

\put(40.1,3.2){$t_5$}


\put(37,16){\line(1,0){2}}
\put(39,16){\line(0,-1){10.1}}\put(38.2,10){$\Breve{L}$}
\put(39,5.9){\vector(-1,0){2}}

\put(0.7, 11.5){(I)}\put(39.5, 11.5){(IX)}\put(34.8, 14.5){(X)}
\put(11, 11.5){(II)}\put(26, 11.5){(VII)}
\put(11, 8.5){(III)}\put(26, 8.5){(VIII)}
\put(18.7, 8.5){(IV)}
\put(18.5, 5.5){(V)}
\put(18.5, 2.5){(VI)}

\put(17,0){\text{Diagram (6)}}

\put(0, -1){{\emph where} \ $t_4=(id_A\otimes \widetilde{F}_{X,Y})\oplus (id_A\otimes \widetilde{F}_{X,Y})$}
\put(1.8,-2) {$t_5=(\widetilde{F}_{X,Y}\otimes id_A)\oplus (\widetilde{F}_{X,Y}\otimes id_B)$}

\end{picture}
\end{sideways}\]

\[\begin{sideways}
\scriptsize\setlength\unitlength{0.5cm}
\begin{picture}(35,22)
\put(3,22){$A((FX)B \oplus   (FY)B)$}
\put(3,19){$A(BFX \oplus   BFY)$}
\put(3,16){$A(B(FX \oplus   FY))$}
\put(3,13){$(AB)(FX \oplus   FY)$}
\put(3,10){$(AB)F(X\oplus   Y)$}
\put(3,7){$A(B(F(X\oplus   Y)))$}
\put(3,4){$A((F(X\oplus   Y))B)$}
\put(3,1){$A((FX\oplus   FY)B)$}

\put(13,22){$A((FX)B) \oplus   A((FY)B)$}
\put(13,19){$A(BFX) \oplus   A(BFY)$}
\put(13,13){$(AB)FX \oplus   (AB)FY$}
\put(13,10){$(F(X\oplus   Y))(AB)$}
\put(13,7){$(F(X\oplus   Y)A)B$}
\put(13,4){$(A(F(X\oplus   Y))B$}
\put(13,1){$(A(FX\oplus   FY))B$}

\put(24,22){$(AFX)B \oplus   (AFY)B$}
\put(24,19){$((FX)A)B \oplus   ((FY)A)B$}
\put(24,13){$(FX)(AB) \oplus   (FY)(AB)$}
\put(24,10){$(FX\oplus   FY)(AB)$}
\put(24,7){$((FX\oplus   FY)A)B$}
\put(24,4){$((FX)A\oplus   (FY)A)B$}
\put(24,1){$(AFX\oplus   AFY)B$}

\put(5.5,19.5){\vector(0,1){2.2}}\put(0,20.5){$id\otimes   (u_{B,X}\oplus  \ u_{B,Y})$}
\put(5.5,16.5){\vector(0,1){2.2}}\put(3.5,17.5){$id\otimes   \Breve{L}$}
\put(5.5,15.8){\line(0,-1){2.2}}\put(5.4,15.8){\line(0,-1){2.2}}
\put(5.5,12.8){\vector(0,-1){2.2}}\put(3,11.5){$id\otimes   \Breve{F}$}
\put(5.5,9.8){\line(0,-1){2.2}}\put(5.4,9.8){\line(0,-1){2.2}}
\put(5.5,6.8){\vector(0,-1){2.2}}\put(1.8,5.5){$id\otimes   u_{B,X\oplus   Y}$}
\put(5.5,1.5){\vector(0,1){2.2}}\put(1.8,2.5){$id\otimes   (\Breve{F}\otimes   id)$}

\put(15.7,19.5){\vector(0,1){2.2}}\put(10.3,20.5){$id\otimes   (u_{B,X}\oplus   u_{B,Y})$}
\put(15.7,13.5){\line(0,1){5.2}}\put(15.8,13.5){\line(0,1){5.2}}
\put(15.7,7.5){\line(0,1){2.2}}\put(15.8,7.5){\line(0,1){2.2}}
\put(15.7,4.5){\vector(0,1){2.2}}\put(11.5,5.5){$u_{A,X\oplus   Y}\otimes  \  id$}
\put(15.7,1.5){\vector(0,1){2.2}}\put(11.5,2.5){$ (id\otimes   \Breve{F})\otimes   \ id$}

\put(26.7,21.8){\vector(0,-1){2.2}}\put(19.5,20.5){$(u_{A,X}\otimes   id)\oplus  \  (u_{A,Y}\otimes   id)$}
\put(26.7,13.5){\line(0,1){5.2}}\put(26.8,13.5){\line(0,1){5.2}}
\put(26.7,10.5){\line(0,1){2.2}}\put(26.8,10.5){\line(0,1){2.2}}
\put(26.7,7.5){\line(0,1){2.2}}\put(26.8,7.5){\line(0,1){2.2}}
\put(26.7,4.5){\line(0,1){2.2}}\put(26.8,4.5){\line(0,1){2.2}}
\put(26.7,1.5){\vector(0,1){2.2}}\put(21,2.5){$(u_{A,X}\oplus   u_{A,Y} )\ \otimes   id$}

\put(8.6,22.1){\vector(1,0){4.2}}\put(10,22.3){$\Breve{L}$}
\put(7.6,19.1){\vector(1,0){5.2}}\put(10,19.3){$\Breve{L}$}
\put(7.6,13.1){\vector(1,0){5.2}}\put(10,13.3){$\Breve{L}$}
\put(7.6,10.1){\vector(1,0){5.2}}\put(9,10.5){$u_{AB,X\oplus   Y}$}
\put(7.6,4.1){\vector(1,0){5.2}}\put(9,4.3){$a=id$}
\put(7.6,1.1){\vector(1,0){5.2}}\put(9,1.3){$a=id$}

\put(19.5,22.1){\vector(1,0){4.2}}\put(20.5,22.4){$id\oplus   id$}
\put(18.5,13.1){\vector(1,0){5.2}}\put(19,13.5){$u_{AB,X}\oplus   u_{AB,Y}$}
\put(23.7,10.1){\vector(-1,0){5.9}}\put(20,10.4){$\Breve{F}\otimes   id$}
\put(23.7,7.1){\vector(-1,0){6}}\put(19.5,7.5){$(\Breve{F}\otimes   id)\otimes   id$}
\put(18,1.1){\vector(1,0){5.7}}\put(20,1.3){$\Breve{L}\otimes   id$}

\put(2.8,16.1){\line(-1,0){3}}
\put(-0.2,16.1){\line(0,-1){9}}
\put(-0.2,7.1){\vector(1,0){2.8}}
\put(0.2,14.5){$id\otimes  (id\otimes   \Breve{F})$}

\put(2.8,22.1){\line(-1,0){4}}
\put(-1.2,22.1){\line(0,-1){21}}
\put(-1.2,1.1){\vector(1,0){4}}
\put(-0.5,1.3){$id\otimes \Breve{R}$}

\put(29.5,4.1){\line(1,0){2.7}}
\put(32.2,4.1){\line(0,1){15}}
\put(32.2,19.1){\vector(-1,0){1.5}}
\put(30,16){$\Breve{R}\otimes id$}

\put(28.8,1.1){\line(1,0){4.2}}
\put(33,1.1){\line(0,1){21}}
\put(33,22.1){\vector(-1,0){3.2}}
\put(33.5,11.5){$\Breve{R}=id$}

\put(0.8,19){(I)}\put(1,11.5){(II)}
\put(8.5,20.5){(III)}\put(9.5,16){(IV)}\put(9.5,7){(VI)}\put(9.5,2.5){(VII)}
\put(15.5,11.5){(V)}
\put(20,16){(VIII)}\put(20,8.5){(IX)}\put(20,4){(X)}
\put(29.6,11.5){(XI)}
\put(30.6,20.5){(XII)}
\put(14,0){\text{Diagram (7)}}
\end{picture}
\end{sideways}\]

\newpage

\end{document}